\documentclass[11pt, reqno]{amsart}
\usepackage{amssymb,latexsym,amsmath,amsfonts}
\usepackage[mathscr]{eucal}

\numberwithin{equation}{section}

\theoremstyle{definition}
\newtheorem{definition}{Definition}[section]

\newtheorem{remark}[definition]{Remark}

\theoremstyle{plain}
\newtheorem{theorem}[definition]{Theorem}
\newtheorem{lemma}[definition]{Lemma}

\newtheorem{corollary}[definition]{Corollary}

\newtheorem{result}[definition]{Result}

\newcommand{\al}{\alpha}

\newcommand\ba[1]{\overline{#1}}
\newcommand\hull[1]{\widehat{#1}}

\newcommand{\bdy}{\partial}

\newcommand{\ball}{\mathbb{B}}
\newcommand{\disc}{\mathbb{D}}

\newcommand{\smoo}{\mathcal{C}}
\newcommand{\hol}{\mathcal{O}}
\newcommand{\poly}{\mathscr{P}}

\newcommand{\rl}{\Re\mathfrak{e}}
\newcommand{\imag}{\Im\mathfrak{m}}

\newcommand{\cplx}{\mathbb{C}}

\newcommand{\rea}{\mathbb{R}}

\begin{document}

\title[Some observation concerning polynomial convexity]{Some observations concerning polynomial convexity}
\author{Sushil Gorai}
\address{Department of Mathematics and Statistics, Indian Institute of Science Education and Research Kolkata,
Mohanpur, Nadia, West Bengal 741246, India}
\email{sushil.gorai@iiserkol.ac.in}

\keywords{Polynomial convexity; closed ball; totally real; Lagrangians}
\subjclass[2010]{Primary: 32E20}
\thanks{This work is partially supported by an INSPIRE Faculty Fellowship (IFA-11MA-02) funded by DST and is also supported by a research 
grant under MATRICS scheme (MTR/2017/000974)}

\begin{abstract}
In this paper we discuss a couple of observations related to polynomial convexity. More precisely, 
\begin{itemize}
\item[(i)] We observe that the union of finitely many disjoint closed balls with centres in $\bigcup_{\theta\in[0,\pi/2]}e^{i\theta}V$ is polynomially convex, where $V$ 
is a Lagrangian subspace of $\cplx^n$. 
\item[(ii)] We show that any compact subset $K$ of $\{(z,w)\in\cplx^2:q(w)=\ba{p(z)}\}$, where $p$ and $q$ are two non-constant holomorphic polynomials 
in one variable, is polynomially convex and $\poly(K)=\smoo(K)$.
\end{itemize}
\end{abstract}

\maketitle

\section{Introduction}\label{S:intro}
For a compact set $K\subset\cplx^n$ the {\em polynomially convex hull} is defined by 
\[
\hull{K}:= \left\{z\in \cplx^n : |p(z)|\leq \sup_{K}|p|,\, p\in \cplx [z_1, \dots, z_n] \right\}.
\]
 $K$ is said to be {\em polynomially convex} if $\hull{K}=K$. Similarly, we define {\em rationally convex hull} of a compact 
 set $K\subset\cplx^n$ as 
 \[
 \hull{K}_R:=\left\{z\in\cplx^n: |f(z)|\leq \sup_K |f|,\, f\;\text{is a rational function}\right\}.
 \]
 $K$ is said to be {\em rationally convex} if $\hull{K}_R=K$. We note that $K\subset \hull{K}_R\subset\hull{K}$.
 Any compact convex subset of $\cplx^n$, $n\geq 1$, is polynomially convex. Thanks to Runge's approximation theorem, 
 any compact subset of $\cplx$ is rationally convex.
 A compact subset $K\subset\cplx$ is polynomially convex if and only if $\cplx\setminus K$ is connected. 
 Hence, in $\cplx$, polynomial convexity becomes a purely topological property on the compact set; of course, the reason 
is the very deep interconnections between topology and complex analysis in one variable. In $\cplx^n$, $n\geq 2$, it is not a topological 
property. In fact, there exist two compact subsets in $\cplx^2$, which are homeomorphic, 
but one of them is polynomially convex and the other is not. For instance, consider the unit circle placed in $\rea^2\subset\cplx^2$ and 
in $\cplx\times\{0\}\subset\cplx^2$. The first circle is polynomially convex while the later is not. 
Polynomial convexity is very closely related with
polynomial approximation. Below we mention a theorem that exhibit such a connection (see Stout's book\cite{stout} for more on these). 
\smallskip

\begin{theorem}[Oka-Weil]
Let $K\subset\cplx^n$ be a compact polynomially convex. Then any function that is holomorphic 
in a neighborhood of $K$ can be approximated uniformly on $K$ by polynomials in $z_1,\dots, z_n$.
\end{theorem}

\noindent Although the questions of polynomial convexity appear naturally in connections with questions in function theory, it 
is, however, very difficult to determine whether a given compact in $\cplx^n$, $n\geq 2$, is polynomially convex. For 
instance, no characterization of a finite union of pairwise disjoint polynomially convex sets is known. 
Characterization is not known even for convex compact sets. The union of two disjoint compact convex sets 
is polynomially convex, thanks to Hahn-Banach separation theorem. The union of three disjoint compact convex set is not necessarily 
polynomially convex (see Kallin \cite{K2}). This leads researchers to focus on certain families of compacts having with some geometrical  
properties in $\cplx^n$ to study the question 
of polynomial convexity. In these paper we present two families of compacts which are polynomially convex. The first one is finite union 
of disjoint closed balls with centres lying in some particular region in $\cplx^n$. Let us now make brief survey about works done about polynomial 
and rational convexity for finite union of pairwise disjoint closed balls.
In the same paper Kallin \cite{K2} showed that the union of three disjoint closed balls is polynomially convex. 
It is an open problem whether the union of four disjoint closed balls in $\cplx^n$, $n\geq 2$, is polynomially convex. The most 
general result in this direction is given 
by Khudaiberganov \cite{Kbg1}. 

\begin{result}[Khudaiberganov]\label{res-kbg}
The union of any finite number of disjoint balls in $\cplx^n$ with centres lying in $\rea^n\subset\cplx^n$ is polynomially convex. 
\end{result}

The question of rational convexity of the union of finitely many disjoint closed balls in $\cplx^n$ is studied by S. Nemirovski\u\i \cite{SN}. He proved that any finite 
union of disjoint closed balls is rationally convex using a result of Duval-Sibony \cite{DS}. 
\smallskip

In this note we report an interesting (at least to the author) 
observation proceeding along the similar argument as Khudaiberganov \cite{Kbg1} (see also \cite{ChirSmir}). Before stating 
the observation we need to recall few basic notions 
in symplectic geometry. We consider $(\cplx^n, \omega_0)$ as a symplectic manifold with the standard symplectic form 
\[
\omega_0=\sum_{j=1}^n dx_j\wedge dy_j.
\]
A linear subspace $V$ of $\cplx^n$ is said to be a Lagrangian subspace of $\cplx^n$ if 
$V=\{ u\in\cplx^n: \omega_0(u,v)=0\;\;\forall v\in V\}$. For a Lagrangian subspace $V$, it follows that for every $\theta\in\rea$,
$e^{i\theta}V:=\{e^{i\theta}v\in\cplx^n:\; v\in V\}$ is also a Lagrangian subspace.
\begin{remark}\label{rmk-kbg}
We note that if a subspace 
$V$ of $\cplx^n$ is Lagrangian, then the image under a
unitary transformation is also a Lagrangian subspace. Also there exists a unitary $T:\cplx^n\to\cplx^n$ such that 
\[
T(V)=\rea^n\subset\cplx^n.
\]  
By Result~\ref{res-kbg} we know that the union of finitely many disjoint closed balls are polynomially convex if 
the centres lie in a Lagrangian subspace of $\cplx^n$. 
\end{remark}

Our first observation is: 

\begin{theorem}\label{thm-unionballs}
Let $V$ be a Lagrangian subspace of $\cplx^n$. The union of finitely many disjoint closed balls
is polynomially 
convex if their centres lie in $\bigcup_{\theta\in [0,\pi/2]} {e^{i\theta}V}$.
\end{theorem}

We now fix some notations: $B(a;r)$ denotes the open ball in $\cplx^n$ centred 
at $a=(a_1,\dots, a_n)$ and with radius $r$, i.e., $B(a;r)=\{z\in\cplx^n: |z_1-a_1|^2+\dots+|z_n-a_n|^2<r^2\}$ 
and $\mathbb{B}$ denotes the open unit ball. Open unit disc in $\cplx$ is denoted by $\disc$.
For a compact $K\subset\cplx^n$, let $\smoo(K)$
denotes the algebra of all continuous function and $\poly(K)$  denotes the closed subalgebra of $\smoo(K)$ generated by polynomials 
in $z_1,\dots, z_n$.

The other class of compact subsets that we consider in this note are subsets lying in certain real analytic variety in $\cplx^2$ 
of the form $\left\{(z,w)\in\cplx^2:q(w)=\ba{p(z)}\right\}$, where $p$ and $q$ are two non-constant 
holomorphic polynomials in one variable.
Our next observation gives  a generalization of Minsker's theorem \cite{Minsker}(see Corollary~\ref{coro-minskertype}). 
Minsker proved that the algebra generated by $z^m$ and $\ba{z}^n$ is dense in $C(\ba{\disc})$ if $\gcd(m,n)=1$.
\smallskip

\begin{theorem}\label{thm-realvariety}
Any compact subset $K$ of $S:=\{(z,w)\in\cplx^2:q(w)=\ba{p(z)}\}$, where $p$ and $q$ are two non-constant holomorphic polynomial 
in one variable, is polynomially convex and $\poly(K)=\smoo(K)$.
\end{theorem}

\noindent If one of $p$ and $q$ is constant a compact patch 
$K=\left\{(z,w)\in\cplx^2:q(w)=\ba{p(z)}\right\}\cap \ba{B(a;r)}$ is polynomially convex 
but $\poly(K)\neq \smoo(K)$.

\section{Technical preliminaries}\label{S:technical}

In this section we mention some results from the literature that will be useful in the proof. The first one is a lemma due to Kallin \cite{K1} (see 
\cite{dP2} for a survey on the use of Kallin's lemma)

\begin{lemma}[Kallin]\label{L:kallin}
 Let $K_1$ and $K_2$ be two compact polynomially convex subsets in $\cplx ^n$.  
 Suppose further that there exists a holomorphic polynomial
$P$ satisfying the following conditions:
 \begin{enumerate}
 \item[$(i)$] $\hull{P(K_1)} \cap \hull{P(K_2)} \subset \{0\}$; and
 \item[$(ii)$] $P^{-1}\{0\} \cap (K_1 \cup K_2)$ is polynomially convex.
 \end{enumerate}
 Then $K_1 \cup K_2$ is polynomially convex.
\end{lemma}
 Next, we mention a basic but nontrivial result from H\"{o}rmander's book \cite{Hor}.
 \begin{result}\cite[Theorem~4.3.4]{Hor}\label{res-holpsh}
 Let $K$ be a compact subset of a pseudoconvex domain $\Omega$ in $\cplx^n$. Then $\hull{K}_\Omega =\hull{K}^P_\Omega$, where 
 $\hull{K}_\Omega =\left\{z\in\Omega: |f(z)|\leq \sup_{w\in K}|f(w)|\;\forall f\in\hol(\Omega)\right\}$ and 
 $\hull{K}^P_\Omega=\left\{ z\in\Omega: u(z)\leq \sup_{w\in K}u(w)\; \forall u\in{\sf psh}(\Omega)\right\}$. 
 \end{result}
 We note that, when $\Omega=\cplx^n$, Result~\ref{res-holpsh} gives us that the polynomially convex hull $\hull{K}$ is equal to the 
 plurisubharmonically convex hull $\hull{K}^P$. It plays a vital role in our proof of Theorem~\ref{thm-realvariety}. The main idea behind 
 our proof of approximation part of Theorem~\ref{thm-realvariety} is to look at the points where the set $S$ is totally real. 
 A real submanifold $M$ of $\cplx^n$ is said to be {\em totally real} at $p\in M$ if $T_pM\cap iT_pM=\{0\}$, where 
$T_pM$ denotes the tangent space of $M$ at $p$ viewed as a subspace in $\cplx^n$. A real submanifold $M$ is said to be {\em totally real} if it 
is totally real at every point $p\in M$.  Following 
 result from \cite{SG4} gives a characterization of a level set of certain 
 map from $\cplx^n$ to $\rea^n$ to be totally real. 
 
 \begin{result}\cite[Lemma~2.5]{SG4}\label{res-totreal}
 Let $\rho_1,\dots, \rho_n$ be real valued functions so that $\rho:=(\rho_1,\dots,\rho_n):\cplx^n\to\rea^n$ is a 
 submersion. The level set $S:=\{z\in\cplx^n: \rho(z)=0\}$ is totally real at a point $p\in S$ if and only if $\det A_p\neq 0$, where 
\[A_p= \begin{pmatrix}
          \dfrac{\bdy \rho_1}{\bdy \ba{z_1}}(p)\;\; \dots\;\; \dfrac{\bdy \rho_1}{\bdy \ba{z_n}}(p)\\
          \dfrac{\bdy \rho_2}{\bdy \ba{z_1}}(p)\;\; \dots \;\;\dfrac{\bdy \rho_2}{\bdy \ba{z_n}}(p)\\
          \vdots\\                                                 
          \dfrac{\bdy \rho_n}{\bdy \ba{z_1}}(p)\;\; \dots \;\;\dfrac{\bdy \rho_n}{\bdy \ba{z_n}}(p)
          \end{pmatrix}
          \]
         \end{result}
 
 \noindent  It is well-known
that any totally-real submanifold in $\cplx^n$ is locally polynomially convex at every point (see  \cite{Wermer1}, \cite{HW}) 
i.e., for each $p\in M$ there exists a ball $B(p;r)$ such that $M\cap \ba{B(p;r)}$ 
is polynomially convex.
We now mention the following approximation result due to O'Farrell, Preskenis and Walsh \cite{OPW} 
for compact sets that are locally contained in totally-real submanifolds of $\cplx^n$.

\begin{result}[O'Farrell-Preskenis-Walsh]\label{res-OPW}
Let $K\subset\cplx^n$ be a compact polynomially convex subset of $\cplx^n$ and $E\subset K$ be such that 
$K\setminus E$ is locally contained in totally-real submanifolds of $\cplx^n$. Then 
\[
\poly(K)=\left\{f\in\smoo(K): f|_E\in\poly(E)\right\}.
\]
\end{result}
Next, we mention another approximation result that will be useful in our proof of Theorem~\ref{thm-realvariety}.
 \begin{result}\cite[Lemma 2.3]{SG4}\label{res-approx}
 Let $K$ be a compact subset of $\cplx^n$ such that $\poly(K)=\smoo(K)$. Then any closed 
 subset $L$ of $K$ is polynomially convex and $\poly(L)=\smoo(L)$.
 \end{result}

\section{Union of balls}\label{sec-unionballs}
Our aim in this section is to prove Theorem~\ref{thm-unionballs}. Before going into the proof
we state and prove a lemma about the image of a ball centred at $\rea^n\subset\cplx^n$ under the polynomial 
$p(z_1,\dots, z_n)=\sum_{j=1}^nz_j^2$. This will play a very crucial role in our proof of Theorem~\ref{thm-unionballs}.
\begin{lemma}\label{lem-imageball}
Let $a\in\rea^n$ and $0\leq r\leq1$ be such that $|a|-r>1$. Then the image of the closed ball $\ba{B(a,r)}$ under 
the polynomial $p(z_1,\dots,z_n)=\sum_{j=1}^nz_j^2$ 
lies in the affine half-space $\{w\in\cplx:\rl w>1\}$. 
\end{lemma}

\begin{proof}
Let $z\in \ba{B(a,r)}$, where $a\in\rea^n$.
Writing $z=x+iy$, $x,y\in\rea^n$, we get that
\begin{equation}\label{eq-realcentre}
|x-a|^2+ |y|^2 \leq r^2.
\end{equation}
For all $z\in \ba{B(a,r)}$ we obtain that
\begin{align*}
\rl p(z) &=|x|^2-|y|^2\\
&\geq |x|^2-r^2+|x-a|^2 \quad\text{(using Equation~\eqref{eq-realcentre})}\\
&= |x|^2-r^2+|x|^2-2\langle x, a\rangle +|a|^2\\
&\geq 2|x|^2-2|x||a|+|a|^2-r^2.
\end{align*}

We now consider the function $\varphi(t)=2t^2-2t|a|+|a|^2-r^2$. 
The function $\varphi(t)$ has a minimum at $t=\dfrac{|a|}{2}$ and is increasing for $t>\dfrac{|a|}{2}$. 
Since, by assumption, $|a|-r>1$ and $0\leq r\leq 1$, we get that 
$r<|a|/2$. This implies that $\dfrac{|a|}{2}<|a|-r$. Therefore, 
for all $t\geq|a|-r$, 
\begin{align}
\varphi(t) &\geq \varphi(|a|-r)\notag\\
& =2(|a|-r)^2-2(|a|-r)|a|+|a|^2-r^2\notag\\
&= |a|^2-2r|a|+r^2\notag\\
&=(|a|-r)^2.\label{eq-phincrease}
\end{align}

For $z\in \ba{B(a,r)}$, $z=x+iy$, we have $|x|\geq |a|-r$.
Hence, in view of Equation~\eqref{eq-phincrease}, we obtain that
\begin{align*}
\rl p(z) &=\varphi(|x|)\\
&\geq \varphi(|a|-r) >1\;\;\forall z\in\ba{B(a,r)}.
\end{align*}
Hence, 
 \[
 p(\ba{B(a;r)})\subset \{w\in\cplx:\rl w>1\}.
\]
\end{proof}

In this section we provide a proof of Theorem~\ref{thm-unionballs}. The main idea behind the proof is due to Khudaiberganov \cite{Kbg1} (see 
also \cite{ChirSmir})

\begin{proof}[Proof of Theorem~\ref{thm-unionballs}]

Since $V$ is a Lagrangian subspace of $\cplx^n$, there exists a unitary transformation $T:\cplx^n\to\cplx^n$ such that $T(V)=\rea^n$. $\cplx$-linearity of 
$T$ gives us
$ T(\lambda V)=\lambda \rea^n$ for all $\lambda\in\cplx$; in particular,
\[
T(e^{i\theta} V)=e^{i\theta}\rea^n.
\]
Since unitary transformations of $\cplx^2$ maps balls to balls, it is enough to consider the disjoint closed balls with centres lying in $\bigcup_{\theta\in[0,\pi/2]} e^{i\theta}\rea^n$.
Without loss of generality we assume that the closed disjoint balls are as follows: $\ba{\mathbb{B}}$, the closed unit ball, and
$\ba{B(a_j;r_j)}$ such that 
$a_j\in \bigcup_{\theta\in[0,\pi/2]} e^{i\theta}\rea^n$ and $0\leq r_j\leq1$, $j=1,\dots, N.$ 
Since the closed balls are pairwise disjoint, we note that 
\begin{equation}\label{eq-centres}
|a_j|-r_j>1\;\;\forall j=1,\dots, N.
\end{equation}


We show that $\ba{\ball}\cup\left(\bigcup_{j=1}^N\ba{B(a_j;r_j)}\right)$ is polynomially convex. We will use the induction on $N$ for that. 
For $N=1$, clearly, $\ba{\ball}\cup \ba{B(a_1;r_1)}$ is polynomially convex for any ball $B(a_1;r_1)$
with $a_1\in \bigcup_{\theta\in[0,\pi/2]} e^{i\theta}\rea^n$ and $\ba{\ball}\cap\ba{B(a_1;r_1)}=\varnothing$.
 As the induction hypothesis we assume that the union $\ba{\ball}\cup\left(\bigcup_{j=1}^{N-1}\ba{B(\al_j;r_j)}\right)$ of 
 $N$ pairwise disjoint closed balls, one of them being  
the closed unit ball and the others being any $(N-1)$ pairwise disjoint balls with centres 
$\al_j\in\bigcup_{\theta\in[0,\pi/2]}e^{i\theta}\rea^n$ and radii $ r_j\leq 1$, 
is polynomially convex.

Assume the compact sets $K_1:=\ba{\ball}$ and $K_2:=\bigcup_{j=1}^N\ba{B(a_j;r_j)}$. Since $K_2$ is a union 
of $N-1$ disjoint balls with centres in $\bigcup_{\theta\in[0,\pi/2]}e^{i\theta}\rea^n$. Without loss of generality assume that 
$r_N\geq r_j$, $j=1,\dots, (N-1)$.
There exists an invertible $\cplx$-affine transformation $S$ on $\cplx^n$ of the form
\[
S(z)=\mu(z+b),
\]
where $\mu, b\in\cplx$,
such that 
\[
S(\ba{B(a_N;r_N)})=\ba{\ball} \;\text{and}\; S(\ba{B(a_j;r_j)})=\ba{B(c_j;s_j)},
\]
where $c_j\in\bigcup_{\theta\in[0,\pi/2]}e^{i\theta}\rea^n$ and $0\leq s_j\leq 1$ for all $j=1,\dots, (N-1)$. 
We also have $|c_j|-s_j>1$ for all $j=1,\dots, N-1$.
By induction hypothesis, 
$\ba{\ball}\cup\left(\bigcup_{j=1}^{N-1} \ba{B(c_j;s_j)}\right)$ is polynomially convex. Hence, $K_2$ is polynomially convex.

We now use Kallin's lemma (Lemma~\ref{L:kallin}) with the polynomial 
\[
p(z_1,\dots,z_n)= z_1^2+\dots+z_n^2.
\]
to show $K_1\cup K_2$ is polynomially convex. 
 Clearly,
\begin{equation}\label{eq-imageunit}
|p(z)|\leq 1\quad \forall z\in K_1.
\end{equation}
Since $a_j\in \bigcup_{\theta\in[0,\pi/2]}e^{i\theta}\rea^n$, we assume that
$a_j=e^{i\theta_j}b_j$, where $b_j\in\rea^n$ and $\theta_j\in[0,\pi/2]$ for all $j=1,\dots, N$. 
We first fix a $j_0: 1\leq j_0\leq N$.
Corresponding to $j_0$ we consider a 
unitary map $T_{j_0}:\cplx^n\to\cplx^n$ defined by 
\[
T_{j_0}(z)=e^{i\theta_{j_0}}z.
\]
Clearly, 
$T_{j_0}(b_{j_0})=a_{j_0}$ and 
$
T_{j_0}(\ba{B(b_{j_0}; r_{j_0})})=\ba{B(a_{j_0};r_{j_0})}.
$
In view of Lemma~\ref{lem-imageball}, we obtain that
\[
\rl p(z)>1 \quad\forall z\in \ba{B(b_{j_0};r_{j_0})}.
\]
Since $p$ is a homogeneous holomorphic polynomial of degree two, we get 
\[
p(T_{j_0}(z))=e^{2i\theta_{j_0}}p(z)\quad\forall z\in \ba{B(b_{j_0};r_{j_0})}.
\]
Hence, we get that 
\[
\rl \left(e^{-2i\theta_{j_0}}p(z)\right)>1 \quad \forall z\in \ba{B(a_{j_0}; r_{j_0})}.
\]
Therefore, the image of $\ba{B(a_{j_0}, r_{j_0})}$ under the polynomial $p$ lies in the half plane
\[
\left\{w\in\cplx: \rl\left(e^{-2i\theta_{j_0}}w\right)>1 \right\}.
\]
Since we have chosen $j_0$ arbitrarily, hence, for each $j=1,\dots, N$, we obtain that 
\[
p(\ba{B(a_j, r_j)})\subset \left\{w\in\cplx: \rl\left(e^{-2i\theta_j}w\right)>1 \right\}=:H_{\theta_j}.
\]
Writing $w=u+iv$ in $\cplx$, we get the half space as 
\[
H_{\theta_j}=\left\{u+iv\in\cplx: u\cos{2\theta_j}+v\sin{2\theta_j}>1 \right\}.
\]
Since the boundary line of $H_{\theta_j}$ is tangent to the unit circle, $H_{\theta_j}\cap \ba{\disc}=\varnothing$. 

We get the image of $K_2$ under the polynomial $p$
\[
p(K_2)\subset \bigcup_{j=1}^N H_{\theta_j}. 
\]
We also obtain that 
\begin{equation}\label{eq-intersectkallin}
\left(\bigcup_{j=1}^N H_{\theta_j}\right)\cap \ba{\disc}=\varnothing.
\end{equation}

We note that 
\[
H_0=\{u+iv\in\cplx: u>1\}\;\; \text{and}\;\; H_{\pi/2}=\{u+iv\in\cplx: u<-1\}, 
\]
and $H_{\theta_j}\subset \{u+iv\cplx: v>0, u^2+v^2>1\}\cup H_0\cup H_{\pi/2}$ for all $j=1,\dots, N$.
Hence, the 
strip $\{u+iv\in\cplx: -1\leq u \leq 1, v\leq 0\}$ does not intersect the union of half spaces $\left(\bigcup_{j=1}^N H_{\theta_j}\right)$. 
Hence, 
we get that $\cplx\setminus \left(\bigcup_{j=1}^N H_{\theta_j}\right)$ is connected. Therefore, in 
view of Equations \eqref{eq-imageunit} and \eqref{eq-intersectkallin}, we conclude 
\[
\hull{p(K_1)}\cap \hull{p(K_2)}=\varnothing.
\] 
All the conditions of Kallin's lemma are satisfied with 
the above polynomial $p$. Hence, $K_1\cup K_2=\bigcup_{j=0}^NB_j$ is polynomially convex. 
\end{proof}

\section{Compact subsets of certain real analytic variety}\label{sec-realvar}
In this section we provide a proof of Theorem~\ref{thm-realvariety}. The idea is to construct a non-negative plurisubharmonic function on 
$\cplx^n$ such that the set $S$ lies on the zero set of that function. 
\begin{proof}[Proof of Theorem~\ref{thm-realvariety}] 
 Let $B$ be a closed ball in $\cplx^2$. If $S\cap B=\varnothing$, then there is nothing to prove.
Therefore, assume $S\cap B\neq \varnothing$. We divide the proof into two steps.
First we show that $S\cap B$ is polynomially convex. In the second step 
we show that any compact subset $K$ of $S$ is polynomially convex and $\poly(K)=\smoo(K)$.
\smallskip

\noindent {\em Step I: To show $S\cap B$ is polynomially convex.}

\noindent Consider the function $\Psi:\cplx^2\to\rea$ defined by
\[
\Psi(z,w)=|\ba{p(z)}-q(w)|^2.
\]
Clearly, $S=\Psi^{-1}\{0\}$.
\smallskip

A simple computation gives us 
\begin{align*}
\dfrac{\bdy^2\Psi}{\bdy z\bdy \ba{z}}(z,w)&= \left|\dfrac{\bdy p}{\bdy z}(z)\right|^2\\
\dfrac{\bdy^2\Psi}{\bdy z\bdy \ba{w}}(z,w)&= 0=\dfrac{\bdy^2\Psi}{\bdy w\bdy \ba{z}}(z,w)\\
\dfrac{\bdy^2\Psi}{\bdy w\bdy \ba{w}}(z,w)&= \left|\dfrac{\bdy q}{\bdy w}(w)\right|^2.
\end{align*}

The Levi-form of $\Psi$:
\begin{align*}
\mathscr{L}\Psi((z,w); (u,v))
&=\left|\dfrac{\bdy P}{\bdy z}(z)\right|^2|u|^2 +\left|\dfrac{\bdy q}{\bdy w}(w)\right|^2|v|^2\\
&\geq 0 \quad\forall (u,v)\in\cplx^2.
\end{align*}
Therefore, $\Psi$ is plurisubharmonic in $\cplx^2$. Hence, $S\cap B$ is plurisubharmonically convex. 
In view of Result~\ref{res-holpsh}, $S\cap B$ is polynomially convex.
\smallskip

\noindent {\em Step II: To show any compact subset $K\subset S$ is polynomially convex and $\poly(K)=\smoo(K)$.}

\noindent The main insight here is to show that off a very small set $S$ is totally real. In this case we show that there 
is a finite set $E\subset S$ such that $S\setminus E$ is locally contained in totally real submanifold of $\cplx^2$. 
We will use Result~\ref{res-totreal} for that. 
In this case the defining function $\rho$ is 
\[
\rho(z,w)=(\rho_1(z,w), \rho_2(z,w)),
\]
 where 
\[
\rho_1(z,w):=\rl(p(z)-q(w))\quad\text{and}\quad \rho_2(z,w):=\imag(-p(z)-q(w))
\]
Let $(z_0,w_0)\in S$.
The matrix 
\begin{align*}
A_{(z_0,w_0)} &= \begin{pmatrix}
\dfrac{\bdy \rho_1}{\bdy \ba{z}}(z_0,w_0)\quad  \dfrac{\bdy \rho_1}{\bdy \ba{w}}(z_0,w_0)\\
&{}\\
\dfrac{\bdy \rho_2}{\bdy \ba{z}}(z_0,w_0)\quad \dfrac{\bdy \rho_1}{\bdy \ba{w}}(z_0,w_0).
\end{pmatrix}\\
&{}\\
&= \begin{pmatrix} 
\dfrac{1}{2}\ba{\dfrac{\bdy p}{\bdy z}(z_0)}\quad -\dfrac{1}{2}\ba{\dfrac{\bdy q}{\bdy w}}(w_0)\\
&{}\\
-\dfrac{i}{2}\ba{\dfrac{\bdy p}{\bdy z}(z_0)}\quad -\dfrac{i}{2}\ba{\dfrac{\bdy q}{\bdy w}(w_0)}
\end{pmatrix}
\end{align*}
  We obtain that 
  $\det A_{(z_0,w_0)}=0$ if and only if $\dfrac{\bdy p}{\bdy z}(z_0)\dfrac{\bdy q}{\bdy w}(w_0)=0$. 
  Consider the set 
  \[
  Z=\left\{(z_0,w_0)\in \cplx^2: \;\; q(w_0)=\ba{p(z_0)},\;\dfrac{\bdy p}{\bdy z}(z_0)\dfrac{\bdy q}{\bdy w}(w_0)=0\right\}=:Z_1\cup Z_2, 
  \]
  where 
  \begin{align*}
  Z_1 &:=\left\{(z_0,w_0)\in\cplx^2:\;\; q(w_0)=\ba{p(z_0)},\; \dfrac{\bdy p}{\bdy z}(z_0)=0\right\}\\
  Z_2 &:=\left\{(z_0,w_0)\in\cplx^2:\;\; q(w_0)=\ba{p(z_0)},\; \dfrac{\bdy q}{\bdy w}(w_0)=0\right\}.
  \end{align*}
  Since $p$ and $q$ are non-constant holomorphic polynomials, the holomorphic polynomials 
  $\dfrac{\bdy p}{\bdy z}$ and $\dfrac{\bdy q}{\bdy w}$ are not identically zero.
  $\det A_{(z_0,w_0)} \neq 0$ gives us that $\rho$ is locally a submersion at $(z_0,w_0)$.
  Hence, both the sets 
  $Z_1$ and $Z_2$ are finite sets. Hence, by Result~\ref{res-totreal}, $S\setminus Z$ is locally contained in totally-real submanifold. 
  
  Let $K$ be any compact subset of $S$. There exists a closed ball $B$ in $\cplx^n$ such that 
  \[
  K\subset S\cap B.
  \]
  Since $S$ is totally-real except finitely many points, in view of Result~\ref{res-OPW}, we obtain that 
  \[
  \poly(S\cap B)=\smoo(S\cap B).
  \]
  Hence, by Result~\ref{res-approx}, we get that $K$ is polynomially convex and $\poly(K)=\smoo(K)$.
                       
\end{proof}


\begin{corollary}\label{coro-minskertype}
The algebra generated by $z^m$ and $\ba{z}^n$, $m,n\in\mathbb{N}$ is dense in $\smoo(\ba{\mathbb{D}})$.
\end{corollary}

\begin{proof}
Let $K:=\{(z^m,\ba{z}^n)\in\cplx^2: z\in\ba{\disc}\}$. We wish to show $\poly(K)=\smoo(K)$.
Consider the set 
\[
S:=\{(z,w)\in\cplx^2: w^m=\ba{z}^n\}.
\]

Clearly, $K$ is a compact subset of $S$. By using Theorem~\ref{thm-realvariety}, we get that $K$ is 
polynomially convex and $\poly(K)=\smoo(K)$.
\end{proof}
\begin{remark}
A special case, when $\gcd(m,n)=1$, of Corollary~\ref{coro-minskertype} gives us Minsker's theorem \cite{Minsker}.
\end{remark}

\end{document}